\documentclass{amsart}
\usepackage{amssymb}
\usepackage{amsmath}
\usepackage{amsthm}
\usepackage{amsbsy}
\usepackage{bm}
\usepackage{tikz}
\usepackage{graphicx}
\newcommand{\cev}[1]{\reflectbox{\ensuremath{\vec{\reflectbox{\ensuremath{#1}}}}}}

\theoremstyle{plain}
\newtheorem{theorem}{Theorem}[section]
\newtheorem{lemma}[theorem]{Lemma}

\newtheorem{proposition}[theorem]{Proposition}

\theoremstyle{definition}
\newtheorem{definition}[theorem]{Definition}

\theoremstyle{plain}
\newtheorem{example}[theorem]{Example}

\theoremstyle{remark}
\newtheorem{remark}[theorem]{Remark}

\theoremstyle{remark}

\newtheorem*{question}{Question}

\begin{document}
\title{Filling of Closed Surfaces}
\author{Bidyut Sanki}
\address{Department of Mathematics\\
         Indian Institute of Science\\
         Bangalore 560012, India}
\email{bidyut.iitk7@gmail.com}
\date{\today}
\begin{abstract}
Let $F_g$ denote a closed oriented surface of genus $g$. A set of simple closed curves is called a filling of $F_g$ if its complement is a disjoint union of discs. The mapping class group $\text{Mod}(F_g)$ of genus $g$ acts on the set of fillings of $F_g$. The union of the curves in a filling forms a graph on the surface which is a so-called decorated fat graph. It is a fact that two fillings of $F_g$ are in the same $\text{Mod}(F_g)$-orbit if and only if the corresponding fat graphs are isomorphic. We prove that any filling of $F_2$ whose complement is a single disc (i.e., a so-called minimal filling) has either three or four closed curves and in each of these two cases, there is a unique such filling up to the action of $\text{Mod}(F_2)$.

We provide a constructive proof to show that the minimum number of discs in the complement of a filling pair of $F_2$ is two. Finally, given positive integers $g$ and $k$ with $(g, k)\neq (2, 1)$, we construct a filling pair of $F_g$ such that the complement is a union of $k$ topological discs.  
\end{abstract}

\maketitle
%%%%%%%%%%%%%% SECTION 1 %%%%%%%%%%%%%%%%%

\section{Introduction}
Suppose $F_g$ is a closed oriented surface of genus $g$ and $X=\{\gamma_i\;|\;i=1, \dots, n\}$ is a nonempty collection  of simple closed curves on $F_g$ such that $i(\gamma_i, \gamma_j)=|\gamma_i\cap\gamma_j|$ for $i\neq j$, i.e., $\gamma_i$ and $\gamma_j$ are in minimal position. Here, $i(\alpha, \beta)$ denotes the geometric intersection number of the closed curves $\alpha$ and $\beta$ (see Section 1.2 in~\cite{FM}). The set $X$ is called a \emph{filling} of the surface $F_g$ if the complement of $\bigcup\limits_{i=1}^{n}\gamma_i$ in $F_g$ is a disjoint union of topological discs. If the complement is a single disc, then we say that $X$ is a \emph{minimal filling}. Further, if the number of curves in $X$ is two, then we say that $X$ is a \emph{filling pair}. The combinatorial complexity of a filling $X$ of $F_g$ ($g\geq 2$), denoted by $T_k(X)$, is the number of simple closed curves which intersect the union of curves in $X$ no more than $k$ times. In~\cite{TS} (Theorem 1.2), Aougab and Huang have proved that $T_1(\alpha, \beta)\leq 4g-2$ and equality holds if and only if $(\alpha, \beta)$ is a minimal filling pair of $F_g$. Fillings of surfaces are well studied in topology and geometry. It has been studied extensively in~\cite{FH},~\cite{SSP},~\cite{APP},~\cite{TS},~\cite{PJP},~\cite{T}. 

In~\cite{FH}, Fanoni and Parlier have studied fillings on punctured surfaces. The authors have shown that any filling set of curves on a surface $S$ with at least one puncture intersecting pairwise at most $k$ times has cardinality $N$, where $N$ is the smallest integer satisfying $kN(N-1)\geq 2|\chi(S)|$ (see Theorem 1, in ~\cite{FH}). Moreover, Fanoni and Parlier have proved the fact that if $\{\gamma_1, \dots, \gamma_M\}$ is a filling set of \emph{systoles} (shortest essential closed geodesics) of a hyperbolic surface of signature $(g, n)$ with systole length $l$, then $4lM\geq 2\pi(2g-1)+\pi(n-2)$ (see Theorem 4 in ~\cite{FH}). 

Schmutz has studied the function $\emph{syst}$ which maps a surface to the length of a systole. In \cite{SSP}, the author has proved that under some regularity conditions $syst$ is a topological Morse function which is invariant with respect to the action of mapping class group and descends to a proper function on moduli spaces. Any critical point of $syst$ has a filling set of systoles (see Theorem A in~\cite{SSP}). In~\cite{APP} (see Theorem 1.1), Anderson, Parlier and Pettet have constructed a sequence of surfaces $S_{g_k}$ of genera $g_k\to \infty$ with filling set of systoles and Bers constant $\geq 2\sqrt{g_k}$. In~\cite{PJP}, the authors have shown that the asymptotic growth rate for the minimal cardinality of simple closed curves on a closed surface of genus $g$ which fill and pairwise intersect no more that $k$ times is $2\sqrt{g}/k$ as $g\to \infty$ and the cardinality of a filling set of systoles is bounded from below by $g/\log(g)$.

Our motivation to study fillings of oriented closed surfaces is the following. The set of all hyperbolic structures on the surface $F_g\ (g \geq 2)$ up to isometry is called the moduli space of genus $g$ and is denoted by $\mathcal{M}_g$. It is a well known and difficult problem to construct a spine of $\mathcal{M}_g$, i.e., a deformation retract of $\mathcal{M}_g$ of minimal dimension. The subset of $\mathcal{M}_g$ consisting of all the surfaces whose systoles fill the surface is called $\mathit{Thurston\ set}$ which is denoted by $\mathcal{X}_g$. In~\cite{WT}, Thurston has proposed $\mathcal{X}_g$ as a candidate spine of $\mathcal{M}_g$ and provided a sketch of a proof that $\mathcal{X}_g$ is a deformation retract, but it is difficult to complete the proof. Moreover, the contractibility, connectivity and dimension of $\mathcal{X}_g$ remain open.

Let $\text{Mod}(F_g)$ denote the mapping class group of $F_g$, the group of all orientation-preserving homeomorphisms  up to isotopy (see Section 2.1 in~\cite{FM}). It is easy to see that if $X=\{\gamma_1, \dots, \gamma_n\}$ is a filling of the surface $F_g$ and $[f]\in \text{Mod}(F_g)$, then the set $[f]\cdot X= \{f\circ\gamma_1, \dots, f\circ\gamma_n\}$ is also a filling of $F_g$ with same number of components in the complement. Let $I(F_g)$ denote the set of all simple closed curves on $F_g$. There is a natural action of $\text{Mod}(F_g)$ on the quotient space $$\mathcal{P}_n(F_g):=I(F_g)^n/[(\gamma_1, \dots, \gamma_n)\sim (\gamma_{\sigma(1)}, \dots, \gamma_{\sigma(n)})],$$ where $\sigma$ is a permutation on $\{1,2,\dots, n\}$. 

From another point of view, we can think the union of the curves in $X$ as a connected graph on the surface which we denote by $G_X(F_g)$. This is a so-called \emph{fat graph} (also called a \emph{ribbon graph}, see Section 2.2 in~\cite{BE}), with all vertices of valence 4. It is a fact that:
\begin{theorem}\label{thm:1}
If $F_g$ is the closed oriented surface of genus $g$ and $X, Y\in \mathcal{P}_n(F_g)$ are two fillings of $F_g$, then they are in the same $\text{Mod}(F_g)$-orbit if and only if $G_X(F_g)$ and $G_Y(F_g)$ are isomorphic.
\end{theorem}

It follows from the Euler's equation that if $(\alpha, \beta)$ is a minimal filling of $F_g$, then $i(\alpha, \beta)=2g-1$. In~\cite{TS} (see Theorem 1.1), the authors have shown that for all $g\neq 2$, there exist minimal filling pairs of $F_g$ and for $g=2$, if $(\alpha, \beta)$ is a filling pair of $F_2$, then $i(\alpha, \beta)\geq 4$ (see Theorem 2.17 in~\cite{TS}).  Therefore, there is no minimal filling pair of $F_2$. In this situation, we have the following question.

\begin{question}
What is the minimum number of curves in a minimal filling of $F_2$?
\end{question}

The theorem stated below answers the question. Moreover, it says that there exists a unique mapping class group orbit of such fillings of $F_2$.
\begin{theorem}\label{thm:2}
\begin{enumerate}
\item There exists a unique $\text{Mod}(F_2)$-orbit of triple $\{\alpha, \beta, \gamma\}$ of curves filling $F_2$ minimally.
\item There exists a unique $\text{Mod}(F_2)$-orbit of quadruple $\{\alpha_i: i=1,\dots ,4\}$ filling $F_2$ minimally.
\end{enumerate}
\end{theorem}
Now, we study filling pairs of $F_g$, $g\geq 2$. Let $(\alpha, \beta)$ be a filling pair of $F_g$. The number of disjoint topological discs in the complement $F_g\setminus (\alpha \cup \beta)$ is denoted by $K_g(\alpha,\beta)$. We define $$K(F_g)=\min\{K_g(\alpha, \beta) |\ (\alpha, \beta)\ \text{is a filling pair of}\ F_g\}.$$ In~\cite{TS}, Aougab has shown that $K(F_g)=1$ if $g\geq 3$ and there exists no minimal filling pair of $F_2$, i.e., $K(F_2)\geq 2.$ In ~\cite{TS} (Lemma 2.5),~\cite{FM} (Section 1.3.2), the authors have shown that there exists a filling pair $(\alpha, \beta)$ such that the complement is a union of four pairwise disjoint topological discs which implies $K(F_2)\leq 4.$ Hence, combining the inequalities, we have followed $2\leq K(F_2)\leq 4.$ We prove the following theorem which gives the exact value of $K(F_2)$.

\begin{theorem}\label{thm:3}
There exists a filling pair $(\alpha, \beta)$ of $F_2$ such that the complement is a disjoint union of two topological discs, and hence $K(F_2)=2.$
\end{theorem}

Now, for a positive integer $k$, we have the following question:

\begin{question} Does there exist a filling pair $(\alpha, \beta)$ of $F_g, g\geq 2$ such that $K_g(\alpha, \beta)=k$?
\end{question}

We prove the following result.

\begin{theorem}\label{thm:4}
For every $k\in \mathbb{N}$ and $g\geq 2$ with $(g, k) \neq (2,1)$, there exists a filling pair $(\alpha_k^g, \beta_k^g)$ of $F_g$ such that the complement $F_g\setminus (\alpha_k^g \cup \beta_k^g)$ is a disjoint union of $k$ topological discs.
\end{theorem}
\noindent \textbf{Organization of the paper:} This paper is organized as follows. In Section 2, we give an introduction to fat graphs. We describe the method of construction of the surface associated with a fat graph. Also, we state a lemma which computes the number of boundary components of the surface. This lemma is useful in the subsequent sections. In Section 3, we study the action of mapping class group on the set of fillings of $F_g$ and prove Theorem~\ref{thm:1}. In Section 4, we focus on the minimal fillings of $F_2$. First, we give an independent proof of the fact that there does not exist a minimal filling pair of $F_2.$ We conclude this section with a proof of Theorem~\ref{thm:2}. In Section 5, we study filling pairs of $F_g$ for $g\geq 2$. We prove  Theorem~\ref{thm:3} and Theorem~\ref{thm:4} in this section.

\noindent \textbf{Acknowledgement:} The author would like to thank to Prof. Siddhartha Gadgil for carefully reading through a draft of this work and for his many helpful suggestions. Also, thanks to the referee for several helpful comments and suggestions.

%%%%%%%%%%%%% SECTION 2 %%%%%%%%%%%%%%%%%%%%%%%

\section{Fat graphs}
Before going to the definition of a fat graph, we recall the definition of a graph. The following definition of a graph is not the standard one used in ordinary graph theory, but we can easily see that this definition is equivalent to the standard definition. 

\begin{definition}
A graph is a triple $G=(E, \sim, \sigma_1)$, where
\begin{enumerate}
\item $E$ is a finite, non-empty set.

\item $\sim$ is an equivalence relation on $E$.

\item $\sigma_1:E \to E$ is a fixed-point free involution.
\end{enumerate}
\end{definition}

The set $E$ is called the set of directed edges. The fixed-point free involution $\sigma_1$ maps a directed edge to its reverse edge, i.e., $$\sigma_1(\vec{e})=\cev{e},\;\;\textit{for all} \;\;  \vec{e}\in E$$ and hence the set $E_1=E/{\sigma_1}$ of all orbits of $\sigma_1$ is the set of all undirected edges. The equivalence relation $\sim$ is defined by following: $$\vec{e}_1 \sim \vec{e}_2 \Leftrightarrow \vec{e}_1,\ \vec{e}_2\ \text{have the same initial vertex.}$$ The set $V=E/{\sim}$ of all equivalence classes of $\sim$ is the set of vertices of the graph. If $\vec{e}\in E$ is a directed edge, then we say $\vec{e}$ is emanating or going out from the vertex $[\vec{e}]$, where $[\vec{e}]$ is the equivalence class of $\sim$ containing $\vec{e}$. For $v\in V$, the degree of $v$ is $deg(v)=|v|.$
 
\begin{definition}
A fat graph (ribbon graph) is a quadruple $G=(E, \sim, \sigma_1, \sigma_0)$, where
\begin{enumerate}
\item $(E, \sim, \sigma_1)$ is a graph.

\item $\sigma_0$ is a permutation on $E$ so that each cycle
corresponds to a cyclic order on the set of oriented edges
going out from a vertex.
\end{enumerate}
\end{definition}

\subsection{Surface associated with a fat graph}
We construct a topological surface with boundary corresponding to a fat graph $G$ as described below. We take a closed disc corresponding to each vertex and a rectangle corresponding to each edge. Then we identify the sides of the rectangles with the boundary of the discs according to the order of the edges incident to a vertex. The local picture at a vertex of degree four is given in Figure~\ref{fig:1}.
\begin{figure}[htbp]
\begin{center}
\begin{tikzpicture}
\draw (0,0) circle [radius=0.5]; \draw (0.45, 0.2) -- (1.5, 0.2); \draw (0.45, -0.2) -- (1.5, -0.2); \draw [->] (1.01, -0.2) -- (1, -0.2); \draw (-0.45, 0.2) -- (-1.5,0.2); \draw [<-] (1.01, 0.2) -- (1, 0.2); \draw [<-] (-1.01, -0.2) -- (-1, -0.2); \draw (-0.45, -0.2) -- (-1.5, -0.2);  \draw [->] (-1.01, 0.2) -- (-1, 0.2);

\draw (0.2, 0.45) --(0.2, 1.5);  \draw (-0.2, 0.45) -- (-0.2, 1.5); \draw (0.2, -0.45)-- (0.2, -1.5); \draw (-0.2, -0.45) -- (-0.2, -1.5); \draw [->] (0.2, 1.1) --(0.2, 1); \draw [<-] (-0.2, 1.1) --(-0.2, 1); \draw [->] (-0.2, -1.1) --(-0.2, -1); \draw [<-] (0.2, -1.1) --(0.2, -1);
\draw [->] [domain=20:275] plot ({0.3*cos(\x)}, {0.25*sin(\x)});

\draw [blue, fill] (-4, 0) circle [radius=0.04]; \draw (-3,0)-- (-5, 0); \draw (-4, 1)-- (-4, -1); \draw [->] [domain=20:215] plot ({-4+0.4*cos(\x)}, {0.35*sin(\x)});
\end{tikzpicture}
\end{center}
\caption{Fat graph locally at a vertex of degree 4.}\label{fig:1}
\end{figure}
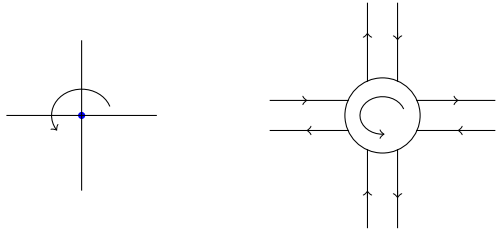

In this way, we obtain the oriented topological surface, denoted by $\Sigma(G)$, corresponding to a given fat graph $G$. Thus, we can talk about the number of boundary components, genus and many other topological notions of a fat graph. We define $\sigma_{\infty}:= \sigma_1 * \sigma_0^{-1}$ and denote the set of orbits of $\sigma_{\infty}$ by $E_{\infty}$.

\begin{lemma}\label{lem:0}
Given a fat graph $G=(E, \sim, \sigma_1, \sigma_0)$, the number of boundary components of the surface $\Sigma(G)$ is the number of orbits of $\sigma_{\infty}$.
\end{lemma}

\begin{example}\label{eg:1}
\noindent Let us consider the fat graph $G=(E, \sim, \sigma_1, \sigma_0)$ (see Figure~\ref{fig:2}), where
\begin{enumerate}
\item $E=\{\vec{e}_i, \cev{e}_i|\ i=1,2,3 \}$ is the set of directed edges.

\item $\sim$ is uniquely determined by the partition $\left\{\{\vec{e_1}, \vec{e_2}, \vec{e_3}\},\{\cev{e}_1, \cev{e}_2, \cev{e}_3 \}\right\}$  of $E$.

\item The involution $\sigma_1:E \to E $ is defined by $\sigma_1(\vec{e}_i)=\cev{e}_i,\; i=1,2,3.$

\item The permutation $\sigma_0: E\to E$ is given by $\sigma_0=(\vec{e}_1, \vec{e}_3, \vec{e}_2)(\cev{e}_1, \cev{e}_2, \cev{e}_3).$
\end{enumerate}

\noindent Therefore, we have $\sigma_{\infty} = (\vec{e}_1, \cev{e}_2)(\vec{e}_3, \cev{e}_1)(\vec{e}_2, \cev{e}_3)$ and hence  the number of boundary components of $G$ is three (Figure~\ref{fig:2}, Lemma~\ref{lem:0}).

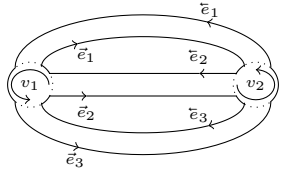
\begin{figure}[htbp]
\begin{center}
\begin{tikzpicture}
\draw [domain= 30:60] plot ({-1.5+0.3*cos(\x)}, {0.3*sin(\x)}); \draw [domain=60:130, dotted] plot ({-1.5+0.3*cos(\x)}, {0.3*sin(\x)}); \draw [domain=130:230] plot ({-1.5+0.3*cos(\x)}, {0.3*sin(\x)}); \draw [domain= 230:300, dotted] plot ({-1.5+0.3*cos(\x)}, {0.3*sin(\x)});\draw [domain= 300:330] plot ({-1.5+0.3*cos(\x)}, {0.3*sin(\x)}); \draw [domain= 330:400, dotted] plot ({-1.5+0.3*cos(\x)}, {0.3*sin(\x)}); \draw [domain= 0:270, ->] plot ({-1.5+0.25*cos(\x)}, {0.2*sin(\x)});

\draw [domain= 50:120, dotted] plot ({1.5+0.3*cos(\x)}, {0.3*sin(\x)}); \draw [domain=120:150] plot ({1.5+0.3*cos(\x)}, {0.3*sin(\x)}); \draw [domain=150:210, dotted] plot ({1.5+0.3*cos(\x)}, {0.3*sin(\x)}); \draw [domain=210:240] plot ({1.5+0.3*cos(\x)}, {0.3*sin(\x)}); \draw [domain=240:310, dotted] plot ({1.5+0.3*cos(\x)}, {0.3*sin(\x)}); \draw [domain=310:410] plot ({1.5+0.3*cos(\x)}, {0.3*sin(\x)}); \draw [domain= 180:450, ->] plot ({1.5+0.25*cos(\x)}, {0.2*sin(\x)});

\draw [->] (1.25, 0.15) --  (0.75, 0.15)node [above] {{\tiny $\cev{e}_2$}}; \draw (0.75, 0.15) -- (-1.25, 0.15); \draw [->](-1.25, -0.15) -- (-0.75, -0.15)node [below] {{\tiny $\vec{e}_2$}}; \draw (-0.75, -0.15) -- (1.25, -0.15);
\draw [domain = 0:60, ->] plot ({1.7*cos(\x)}, {0.23+0.7*sin(\x)}); \draw [domain = 60:180] plot ({1.7*cos(\x)}, {0.23+0.7*sin(\x)}); \draw (0.9, 1) node {{\tiny $\cev{e}_1$}};

 \draw [domain = 180:240, ->] plot ({1.7*cos(\x)}, {-0.23+0.7*sin(\x)}); \draw [domain = 240:360] plot ({1.7*cos(\x)}, {-0.23+0.7*sin(\x)}); \draw (-0.9, -1) node {{\tiny $\vec{e}_3$}};

\draw [domain = 0:130] plot ({1.35*cos(\x)}, {0.24+0.4*sin(\x)}); \draw [domain = 130:180, <-] plot ({1.35*cos(\x)}, {0.24+0.4*sin(\x)}); \draw (-0.75, 0.4) node {{\tiny $\vec{e}_1$}}; 

\draw [domain = 0:50, ->] plot ({1.35*cos(\x)}, {-0.24-0.4*sin(\x)}); \draw [domain = 50:180] plot ({1.35*cos(\x)}, {-0.24-0.4*sin(\x)}); \draw (0.75, -0.4) node {{\tiny $\cev{e}_3$}}; 

\draw (1.5, 0)node {{\tiny $v_2$}}; \draw (-1.5, 0)node {{\tiny $ v_1$}};
\end{tikzpicture}
\end{center}
\caption{The fat graph $G$.}\label{fig:2}
\end{figure}

\end{example}

\begin{remark}
In Example~\ref{eg:1}, if we consider $\bar{\sigma}_0=(\vec{e}_1, \vec{e}_2, \vec{e}_3)(\cev{e}_1, \cev{e}_2, \cev{e}_3)$ instead of $\sigma_0$, then the number of boundary components of the new fat graph is one as $$\bar{\sigma}_{\infty}=\sigma_1*\bar{\sigma}_0^{-1}=(\vec{e}_1, \cev{e}_3, \vec{e}_2, \cev{e}_1, \vec{e}_3, \cev{e}_2).$$ Therefore, we can conclude, on an ordinary graph we can have different fat graph structures.
\end{remark}

\begin{definition}
A fat graph is called decorated if the degree of each vertex is even and at least 4.
\end{definition}

\begin{definition}
A simple cycle in a decorated fat graph is called standard cycle if every two consecutive edges in the cycle are opposite to each other in the cyclic order on the set of edges incident at their common vertex. If a cycle is not standard, we call it as non-standard.
\end{definition}

\begin{definition}\label{def:iso}
Let $G=(E, \sim, \sigma_1, \sigma_0)$ and $G'=(E', \sim', \sigma'_1, \sigma'_0)$ be two fat graphs. $G$ and $G'$ are said to be isomorphic if there exists a bijective function $f: E\to E'$ such that: 
\begin{enumerate}
\item For $x_1, x_2 \in E$, $x_1\sim x_2\  \text{if and only if}\ f(x_1)\sim' f(x_2).$

\item The following diagram commutes:

\begin{center}
\begin{tikzpicture}
\draw (-1.3, 0.7)node {$E$}; \draw (-1.3, -0.7)node {$E$};\draw (1.3, 0.7)node {$E'$};\draw (1.3, -0.7)node {$E'$}; \draw [->] (-0.9,0.7)--(0.9, 0.7); \draw (0, 0.7) node [above] {$f$}; \draw [->] (-0.9,-0.7)--(0.9, -0.7); \draw (0, -0.7) node [below] {$f$}; \draw [->] (-1.3, 0.45) -- (-1.3, -0.45); \draw (-1.3,0)node [left] {$\sigma_1$}; \draw [->] (1.3, 0.45) -- (1.3, -0.45); \draw (1.3,0)node [right] {$\sigma_1'$};
\draw [<-, domain=-20:160] plot ({0.3*cos(\x)},{0.3*sin(\x)});
\end{tikzpicture}
\end{center}
\item $(x_1, x_2,\dots, x_n)$ is the cyclic order on the set of edges going out from $v=\{x_1, x_2,\dots, x_n\}$ if and only if $(f(x_1), f(x_2),\dots, f(x_n))$ is the cyclic order on the set of edges going out from $v'=\{f(x_1), f(x_2),\dots, f(x_n)\}$. 
\end{enumerate}
\end{definition}

\begin{remark}
Condition (3) in Definition~\ref{def:iso} is as the following commutative diagram.
\begin{center}
\begin{tikzpicture}
\draw (-1.3, 0.7)node {$E$}; \draw (-1.3, -0.7)node {$E$};\draw (1.3, 0.7)node {$E'$};\draw (1.3, -0.7)node {$E$}; \draw [->](-0.9,0.7)--(0.9, 0.7); \draw (0, 0.7) node [above] {$f$}; \draw [->] (-0.9,-0.7)--(0.9, -0.7); \draw (0, -0.7) node [below] {$f$}; \draw [->] (-1.3, 0.4) -- (-1.3, -0.4); \draw (-1.3,0)node [left] {$\sigma_0$}; \draw [->] (1.3, 0.4) -- (1.3, -0.4); \draw (1.3,0)node [right] {$\sigma_0'$};
\draw [<-, domain=-20:160] plot ({0.3*cos(\x)},{0.3*sin(\x)});
\end{tikzpicture}
\end{center}
\end{remark}

%%%%%%%%%%%%%%%% SECTION 3 %%%%%%%%%%%%%%%%%%%%%

\section{Mapping class group orbits of fillings}
Let $I(F_g)$ denote the set of all isotopy classes of simple closed curves on $F_g.$ For $n\in \mathbb{N}$, $I(F_g)^n$ denotes the Cartesian product of $n$ copies of $I(F_g$). We define a relation $\sim$ on $I(F_g)^n$ by $(\gamma_1, \gamma_2, \dots, \gamma_n)\sim (\gamma_{\sigma(1)}, \gamma_{\sigma(2)}, \dots, \gamma_{\sigma(n)})$, where $\sigma$ is a permutation on $\{1,2, \dots, n\}$. The mapping class group $\text{Mod}(F_g)$ acts on $I(F_g)^n$ by $[f]\cdot (\gamma_1, \gamma_2, \dots, \gamma_n) = (f\circ\gamma_1,  \dots, f\circ\gamma_n).$ This action descends to an action on the quotient space $\mathcal{P}_n(F_g):=I(F_g)^n/{\sim}.$ Let $X=(\alpha_1, \cdots, \alpha_n)\in \mathcal{P}_n(F_g)$ be a filling of $F_g$ and $[f] \in \text{Mod}(F_g)$, then $[f]\cdot X=(f\circ\alpha_1, \cdots, f\circ\alpha_n)$ is also a filling of $F_g$ with the same number of topological discs in the complement.

\noindent Suppose $X=\{\alpha_1, \dots, \alpha_n\}$ is a filling of $F_g$. Then we can think of the union of the curves in $X$ as a 4-regular decorated fat graph denoted by $G_X(F_g)$ which is described below.
\begin{enumerate}
\item The intersection points $\alpha_i \cap \alpha_j,i\neq j \in \{1,\dots, n\}$ are the vertices.
\item  The sub-arcs of $\alpha_i$'s joining the vertices are the edges.
\item The cyclic order on the set of edges incident at each vertex is uniquely determined by the orientation of the surface. 
\end{enumerate}
Conversely, suppose $G$ is a given decorated 4-regular fat graph with standard cycles $C_i,\;i= 1, \dots, n$. We obtain the closed surface $F(G)$ by capping each boundary component of $\Sigma(G)$ by a topological disc. Then $X_G=\{C_i|i=1, \dots, n\}$ is a filling of $F(G)$.

\begin{proof}[Proof of Theorem~\ref{thm:1}]
Suppose $X$ and $Y$ are in the same $\text{Mod}(F_g)$-orbit which implies that there exists a mapping class $[f] \in \text{Mod}(F_g)$ such that $Y=f\!\cdot\! X$. The restriction $\tilde{f}=f|_{G_X(F_g)}$ of the homeomorphism $f: F_g \to F_g$  gives a fat graph isomorphism $$\tilde{f}: G_X(F_g) \to G_Y(F_g).$$
Conversely, if $\tilde{f}:G_X(F_g) \to G_Y(F_g)$ is an isomorphism, then it can be extended to a homeomorphism $f: F_g \to F_g$ such that $Y=f\cdot X$ which implies that $X$ and $Y$ are in the same $\text{Mod}(F_g)$-orbit.
\end{proof}

%%%%%%%%%%%% SECTION 4 %%%%%%%%%%%%%%%%%%%%%%%%%%%%%

\section{Fillings of $F_2$}
\begin{lemma}\label{lem1}
Let G be a $4$-regular decorated fat graph with three vertices and two standard cycles. Then the number of boundary components in $G$ is at least two.
\end{lemma}
\begin{proof}
Let $C_i,\; i=1,2,$ be the standard cycles of $G$. Then each $C_i$ is simple and consists of three edges. There are three such $4$-regular fat graphs (up to isomorphism) with three vertices and two standard cycles which are denoted by $H_j,\; j=1,2,3$ (see Figure~\ref{fig:3}). For each $j\in \{1,2,3\}$, the graph $H_j$ has three boundary components.

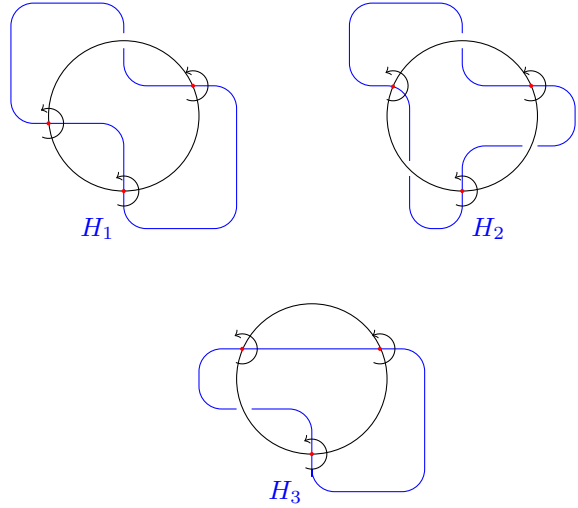
\begin{figure}[htbp]
\begin{center}
\begin{tikzpicture}
\draw (-2,0) circle [radius =1];
\draw [blue, rounded corners=3mm](-2, 0.9) -- (-2, 0.4) -- (-0.5, 0.4) -- (-0.5, -1.5) -- (-2, -1.5)node [left] {$H_1$} -- (-2, -0.1) -- (-3.5, -0.1) -- (-3.5, 1.5) -- (-2, 1.5) -- (-2, 1.1); \draw [fill, red] (-3, -0.1) circle [radius=0.02]; \draw [fill, red] (-2, -1) circle [radius=0.02]; \draw [fill, red] (-1.08, 0.4) circle [radius=0.02];
\draw [->, domain=245:480] plot ({-2+0.2*cos(\x)}, {-1+0.2*sin(\x)}); \draw [->, domain=245:480] plot ({-3+0.2*cos(\x)}, {-0.1+0.2*sin(\x)}); \draw [->, domain=245:480] plot ({-1.08+0.2*cos(\x)}, {0.4+0.2*sin(\x)});

\draw (2.5,0) circle [radius =1];
\draw [blue, rounded corners=3mm] (2.5, 0.9) -- (2.5, 0.4) -- (4, 0.4) -- (4, -0.4) -- (3.5, -0.4); \draw [blue, rounded corners=3mm] (3.3, -0.4) -- (2.5, -0.4) -- (2.5, -1.5)node [right] {$H_2$} -- (1.8, -1.5) -- (1.8, -0.8); \draw [blue, rounded corners=3mm] (1.8, -0.6) -- (1.8, 0.4) -- (1, 0.4) -- (1, 1.5) -- (2.5, 1.5) -- (2.5, 1.1);
 \draw [fill, red] (2.5, -1) circle [radius=0.02];  \draw [fill, red] (3.415, 0.4) circle [radius=0.02]; \draw [fill, red] (1.58, 0.39) circle [radius=0.02];
\draw [->, domain=245:480] plot ({2.5+0.2*cos(\x)}, {-1+0.2*sin(\x)});\draw [->, domain=245:480] plot ({3.41+0.2*cos(\x)}, {0.4+0.2*sin(\x)});\draw [->, domain=245:480] plot ({1.58+0.2*cos(\x)}, {0.4+0.2*sin(\x)});

\draw (-2+2.5,-3.5) circle [radius =1];
\draw [blue, rounded corners=3mm] (-2+1.7, -3.9) -- (-2+2.5, -3.9) -- (-2+2.5, -4.5) -- (-2+2.5, -5)node [left] {$H_3$} -- (-2+4, -5) -- (-2+4, -3.1) -- (-2+1, -3.1) -- (-2+1, -3.9) -- (-2+1.5, -3.9);  \draw [fill, red] (-2+3.405, -3.1) circle [radius=0.02];  \draw [fill, red] (-2+1.575, -3.1) circle [radius=0.02];  \draw [fill, red] (-2+2.5, -4.5) circle [radius=0.02];

\draw [->, domain=245:480] plot ({-2+2.5+0.2*cos(\x)}, {-4.5+0.2*sin(\x)});\draw [->, domain=245:480] plot ({-2+1.575+0.2*cos(\x)}, {-3.1+0.2*sin(\x)});\draw [->, domain=245:480] plot ({-2+3.405+0.2*cos(\x)}, {-3.1+0.2*sin(\x)});

\end{tikzpicture}
\end{center}
\caption{The graphs $H_j, j=1,2,3.$}\label{fig:3}
\end{figure}

\end{proof}

\begin{proposition}
There exists no minimal filling pair of $F_2$. 
\end{proposition}
\begin{proof}
Suppose there is a minimal filling pair $(\alpha, \beta)$ of $F_2$. Then, we define $G:=\alpha \cup \beta.$ We regard $G$ as a decorated fat graph where the intersection points of $\alpha$ and $\beta$ are the vertices and the sub-segments of $\alpha$ and $\beta$ joining two vertices are the edges. The cyclic order on the set of edges incident at each vertex is uniquely determined by the orientation of the surface. 

In another way, we can think $G$ as the 1-skeleton of a cellular decomposition of $F_2$. In the cell decomposition, the number of 0-cells is $i(\alpha, \beta)$. The valency condition implies that the number of 1-cells is $2i(\alpha, \beta)$ and from the minimality condition, we have the number of 2-cells is $1$. Therefore, the Euler characteristic argument implies that $i(\alpha, \beta)= 3.$ Hence, $G$ is a 4-regular decorated fat graph with three vertices, two standard cycles and a single boundary component which contradicts Lemma~\ref{lem1}.
\end{proof}

\begin{proof}[Proof of Theorem~\ref{thm:2}]
\textbf{(1) Filling triple.} Suppose $\{\alpha, \beta, \gamma\}$ is a minimal filling triple of $F_2.$ We define $G:=\alpha \cup \beta \cup  \gamma.$ Then G is a 4-regular decorated fat graph on $F_2$ with three standard cycles $\alpha, \beta$ and $\gamma$. In another point of view, $G$ is the 1-skeleton of a cellular decomposition of $F_2$. The minimality of the filling, Euler characteristic argument and regularity of the graph imply that the number of vertices and edges are $3$ and $6$ respectively. Therefore, to prove the theorem it suffices to prove that there exists a unique $4$-regular decorated fat graph $G$ with  three vertices, three standard cycles and a single boundary component.

\noindent Let $C_i,\; i=1,2,3$, be the standard cycles of $G$. There are two cases to be considered.

\begin{large}
\noindent Case 1.
\end{large}
In this case, we consider that each standard cycle $C_i$ consists of two edges. Up to isomorphism, there are only two distinct such fat graphs $H_1$ and $H_2$ which are given in Figure~\ref{fig:4}.
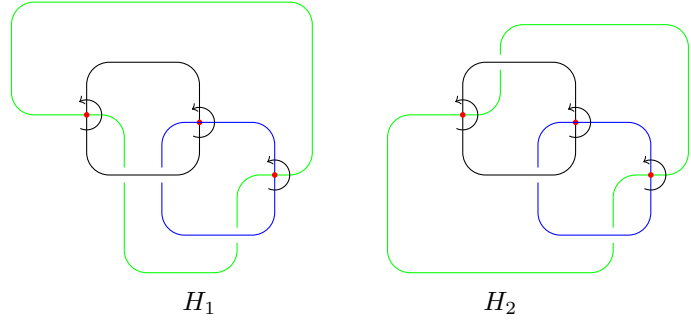
\begin{figure}[htbp]
\begin{center}
\begin{tikzpicture}
\draw [rounded corners=3mm] (0, 0) -- (1.5, 0) -- (1.5, 1.5) -- (0, 1.5)--cycle;
\draw [red, fill] (1.5, 0.7) circle [radius=0.03];
\draw [blue, rounded corners=3mm] (1, -0.1) -- (1, -0.8) -- (2.5, -0.8) -- (2.5, 0.7)--(1,0.7)-- (1, 0.1); 
\draw [green, rounded corners=3mm] (0.5, 0.1) -- (0.5, 0.8) -- (-1, 0.8) -- (-1, 2.3) -- (3, 2.3) -- (3, 0) -- (2,0)-- (2, -0.7); \draw [green, rounded corners=3mm] (2, -0.9) -- (2, -1.3) -- (0.5, -1.3) -- (0.5, -0.1);
\draw [red, fill] (0, 0.8) circle [radius=0.03];
\draw [red, fill] ( 2.5, 0) circle [radius=0.03];
\draw [->, domain=245:480] plot ({0.2*cos(\x)}, {0.8+0.2*sin(\x)});
\draw [->, domain=245:480] plot ({2.5+0.2*cos(\x)}, {0.2*sin(\x)});
\draw [->, domain=245:480] plot ({1.5+0.2*cos(\x)}, {0.7+0.2*sin(\x)});

\draw [rounded corners=3mm] (5, 0) -- (6.5, 0) -- (6.5, 1.5) -- (5, 1.5)--cycle;
\draw [red, fill] (4+2.5, 0.7) circle [radius=0.03];
\draw [blue, rounded corners=3mm] (4+2, -0.1) -- (4+2, -0.8) -- (7.5, -0.8) -- (7.5, 0.7)--(6,0.7)-- (6, 0.1);
\draw [green, rounded corners=3mm] (5.5, 1.4) -- (5.5, 0.8) -- (4,0.8) -- (4, -1.3) -- (7, -1.3) -- (7, -0.9); \draw [green, rounded corners=3mm] (7, -0.7) -- (7, 0) -- (8, 0) -- (8, 2) -- (5.5, 2) -- (5.5, 1.6);
\draw [red, fill] (5, 0.8) circle [radius=0.03];
\draw [red, fill] (7.5, 0) circle [radius=0.03];

\draw (1.5, -1.7) node {$H_1$}; \draw (5.5, -1.7) node {$H_2$};
\draw [->, domain=245:480] plot ({4+2.5+0.2*cos(\x)}, {0.7+0.2*sin(\x)});
\draw [->, domain=245:480] plot ({5+0.2*cos(\x)}, {0.8+0.2*sin(\x)}); \draw [->, domain=245:480] plot ({7.5+0.2*cos(\x)}, {+0.2*sin(\x)});

\end{tikzpicture}
\end{center}
\caption{The graphs $H_j, j=1,2.$}\label{fig:4}
\end{figure}
The fat graph $H_j$, for each $j=1,2$, has three boundary components . Therefore, this case is not possible.

\begin{large}
\noindent Case 2.
\end{large}
In this case, we assume that there is a cycle $C_i$ which consists of three edges. Without loss of generality, we assume that $C_1$ consists of three edges. The Euler characteristic argument implies that other two standard cycles $C_2$ and $C_3$ consist of two edges and one edge respectively. The similar argument as in Case 1 follows that up to isomorphism there are two distinct such fat graphs $\Gamma_1$ and $\Gamma_2$ which are given in Figure~\ref{fig:5}.
\begin{figure}[htbp]
\begin{center}
\begin{tikzpicture}
\draw [rounded corners=3mm] (-1, 0) -- (1.5, 0) -- (1.5, 1.5) -- (-1, 1.5)--cycle;
\draw [red, fill] (1.5, 0.7) circle [radius=0.03];
\draw [->, domain=245:480] plot ({1.5+0.2*cos(\x)}, {0.7+0.2*sin(\x)});
\draw [blue, rounded corners=3mm] (1, -0.1) -- (1, -0.8) -- (2.5, -0.8) -- (2.5, 0.7)--(1,0.7)-- (1, 0.1); 
\draw [green, rounded corners=4mm] (0, -0.7) -- (-1.5,-0.7) -- (-1.5, 2) -- (0, 2)--cycle;
\draw [red, fill] (0, 0) circle [radius=0.03];
\draw [->, domain=245:480] plot ({0.2*cos(\x)}, {0.2*sin(\x)});
\draw [red, fill] (0, 1.5) circle [radius=0.03];
\draw [->, domain=245:480] plot ({0.2*cos(\x)}, {1.5+0.2*sin(\x)});

\draw [rounded corners=3mm] (4, 0) -- (6.5, 0) -- (6.5, 1.5) -- (4, 1.5)--cycle;
\draw [red, fill] (4+2.5, 0.7) circle [radius=0.03];
\draw [->, domain=245:480] plot ({6.5+0.2*cos(\x)}, {0.7+0.2*sin(\x)});
\draw [blue, rounded corners=3mm] (4+2, -0.1) -- (4+2, -0.8) -- (7.5, -0.8) -- (7.5, 0.7)--(6,0.7)-- (6, 0.1);

\draw [green, rounded corners=3mm] (4.5,1.4) -- (4.5, -0.7) -- (5.5, -0.7) -- (5.5, -0.1); \draw [green, rounded corners=3mm] (5.5, 0.1) -- (5.5, 2) -- (4.5, 2) -- (4.5, 1.6); \draw [red, fill] (5.5, 1.5) circle [radius=0.03]; \draw [->, domain=245:480] plot ({5.5+0.2*cos(\x)}, {1.5+0.2*sin(\x)});\draw [red, fill] (4.5, 0) circle [radius=0.03]; \draw [->, domain=245:480] plot ({4.5+0.2*cos(\x)}, {0.2*sin(\x)});
\draw (5.5, -1.2) node {$\Gamma_2$};
\draw (1, -1.2) node {$\Gamma_1$};
\end{tikzpicture}
\end{center}
\caption{The graphs $\Gamma_j, j=1,2.$}\label{fig:5}
\end{figure}
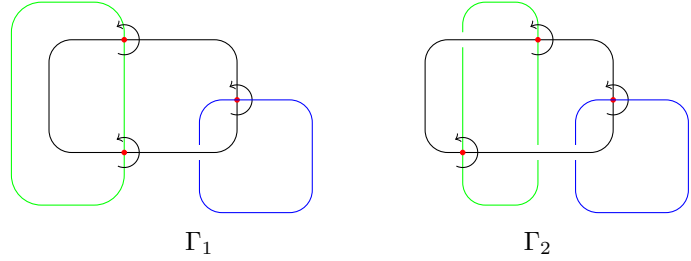
The graph $\Gamma_1$ has three boundary components and $\Gamma_2$ has one boundary component. Therefore, by construction $G=\Gamma_2$ is the unique fat graph which satisfies the theorem. 

\noindent\textbf{(2) Filling quadruple.} To prove the second part, it suffices to show that there exists a unique 4-regular decorated fat graph $G$ with three vertices, four standard cycles and a single boundary component. Suppose $C_i,\ i = 1,2,3,4$, are the standard cycles of $G$. As in the proof of the first part, there are two cases to be considered.

\begin{large}
\noindent Case 1.
\end{large}
Suppose there is a standard cycle consists of three edges. So, without loss of generality, we assume that $C_1$ has length 3. The number of edges in the graph is six. Therefore, it follows that for each $i=2,3,4$, $C_i$ is a loop. There are three vertices on $C_1$ and $C_i, \ i=2,3,4$, are the loops at the vertices on $C_1$. Such a graph, denoted by $H$, is uniquely determined (up to isomorphism) and given in Figure~\ref{fig:6}. The number of boundary components in $H$ is three (see Figure~\ref{fig:6}). Therefore, this case is not possible.
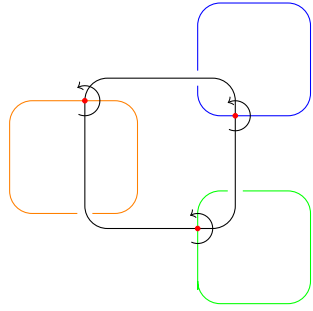
\begin{figure}[htbp]
\begin{center}
\begin{tikzpicture}
\draw [rounded corners=3mm] (0,0) -- (2, 0) -- (2,2) -- (0, 2)--cycle; 
\draw [blue, rounded corners=3mm] (1.5, 1.9) -- (1.5,1.5) -- (3, 1.5) -- (3, 3) -- (1.5, 3) -- (1.5, 2.1); \draw [fill, red] (2, 1.5) circle [radius=0.03]; \draw [->, domain=245:480] plot ({2+0.2*cos(\x)}, {1.5+0.2*sin(\x)}); 
\draw [green, rounded corners=3mm] (1.9, 0.5) -- (1.5, 0.5) -- (1.5, -1) -- (1.5, -1) -- (3, -1) -- (3, 0.5) -- (2.1,0.5); \draw [red, fill] (1.5, 0) circle [radius=0.03]; \draw [->, domain=245:480] plot ({1.5+0.2*cos(\x)}, {0.2*sin(\x)}); \draw [orange, rounded corners=3mm] (0.1, 0.2) -- (0.7, 0.2) -- (0.7, 1.7) -- (-1, 1.7) -- (-1, 0.2) -- (-0.1, 0.2); \draw [red, fill] (0, 1.7) circle [radius=0.03]; \draw [->, domain=245:480] plot ({0.2*cos(\x)}, {1.7+0.2*sin(\x)});

\end{tikzpicture}
\end{center}
\caption{The graph $H$.}\label{fig:6}
\end{figure}

\begin{large}
\noindent Case 2.
\end{large}
In this case, we assume that there is no standard cycle of length three. The only possibility is the following: there are two cycles of length two and two cycles of length one. Such a 4-regular graph $K$ is uniquely determined (up to isomorphism) and given in Figure~\ref{fig:7} which has a single boundary component. Hence, $G=K$ is the unique 4-regular decorated fat graph which satisfies the theorem. 
\begin{figure}[htbp]
\begin{center}
\begin{tikzpicture}
\draw [rounded corners=3mm] (0,0) -- (2.5, 0) -- (2.5, 1.5) -- (0, 1.5)--cycle;
\draw [green, rounded corners=3mm] (1.8, 0.1) -- (1.8, 1) -- (4.3, 1) -- (4.3, -0.5) -- (1.8, -0.5) -- (1.8, -0.1); \draw [red, fill] (0, 1) circle [radius=0.03]; \draw [->, domain=245:480] plot ({0.2*cos(\x)}, {1+0.2*sin(\x)}); \draw [red, fill] (2.5, 1) circle [radius=0.03]; \draw [->, domain=245:480] plot ({2.5+0.2*cos(\x)}, {1+0.2*sin(\x)});
\draw [blue, rounded corners=3mm] (0.7, -0.1) -- (0.7, -0.5) -- (-1.8, -0.5) -- (-1.8, 1) -- (0.7, 1) -- (0.7, 0.1);
\draw [orange, rounded corners=3mm] (3.8, 1.1) -- (3.8, 2) -- (3.8+2.5, 2) -- (3.8+2.5, 0.5) -- (3.8, 0.5) -- (3.8, 0.9); \draw [red, fill] (4.3, 0.5) circle [radius=0.03]; \draw [->, domain=245:480] plot ({4.3+0.2*cos(\x)}, {0.5+0.2*sin(\x)});

\end{tikzpicture}
\end{center}
\caption{The graph $K$.}\label{fig:7}
\end{figure}
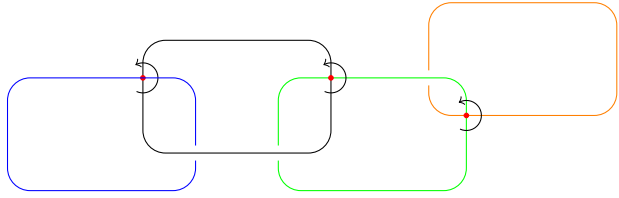

\end{proof}
\begin{remark}
For $n\geq 5$, there does not exist a connected 4-regular decorated fat graph with $n$ standard cycles and three vertices. Hence, there exists no minimal filling $X=\{\gamma_i|\ i=1,2,\dots n\}$ of $F_2$ if $n\geq 5$.
\end{remark}

\begin{proof}[Proof of Theorem~\ref{thm:3}]
Let $(\alpha, \beta)$ be a filling pair of $F_2$ such that the complement is a disjoint union of two  topological discs. Consider the graph $G=\alpha \cup \beta$ on $F_2$. It follows from the Euler characteristic argument that the number of vertices in $G$ is four and the number of edges is eight. The standard cycles of $G$ are the cycles corresponding to the closed curves $\alpha$ and $\beta$. The number of boundary components of $G$ is same as the number of components in the complement of $\alpha\cup\beta$ in $F_2$ which is equal to two.

The proof boils down to showing the existence of a 4-regular graph $G$ as above. Note that, if we have such a graph, we cap the boundary components by topological discs and obtain $F_2$. Consider the graph $G=(E, \sim, \sigma_1, \sigma_0)$ described below (see Figure~\ref{fig:8}).
\begin{enumerate}
\item $E=\{\vec{e}_i, \cev{e}_i, \vec{f}_i, \cev{f}_i| \ i=1,\dots,4\}$.

\item $\sim$ is uniquely determined by the partition $V=\{v_i|i=1,\dots, 4\}$ of $E$, where $v_1=\{\vec{e}_1, \vec{f}_1, \cev{e}_4, \cev{f}_4\}$, $v_2=\{\vec{e}_2, \vec{f}_2, \cev{e}_1, \cev{f}_1\}$, $v_3=\{\vec{e}_3, \cev{f}_3, \cev{e}_2, \vec{f}_4\}$, $v_4=\{\vec{e}_4, \cev{f}_2, \cev{e}_3, \vec{f}_3\}.$

\item $\sigma_1(\vec{e}_i)= \cev{e}_i,\;\; \sigma_1 (\vec{f}_i)= \cev{f}_i, \; i=1,\dots,4.$
 
\item $\sigma_0= (\vec{e}_1, \vec{f}_1, \cev{e}_4, \cev{f}_4)(\vec{e}_2, \vec{f}_2, \cev{e}_1, \cev{f}_1)(\vec{e}_3, \cev{f}_3, \cev{e}_2, \vec{f}_4)(\vec{e}_4, \cev{f}_2, \cev{e}_3, \vec{f}_3).$
\end{enumerate}
\begin{figure}[htbp]
\begin{center}
\begin{tikzpicture}
\draw [rounded corners= 4mm] (2, 1) -- (-2, 1) -- (-2, -1) -- (2, -1) -- cycle;  \draw [red, rounded corners= 4mm] (-1, -0.9) --  (-1,0) -- (-3, 0) -- (-3,-2) -- (3, -2) -- (3, 0) -- (1, 0) -- (1, 0.9); \draw [red, rounded corners= 4mm] (-1, -1.1) -- (-1, -1.6) -- (0, -1.6) -- (0, 1.6) -- (1, 1.6) -- (1, 1.1); \draw [fill] (0,-1) circle [radius=0.04]; \draw [fill] (0,1) circle [radius=0.04]; \draw [fill] (2, 0) circle [radius=0.04]; \draw [fill] (-2, 0) circle [radius=0.04]; \draw [->] (2, 0.70) -- (2, 0.701) node [right] {$\vec{e}_1$}; \draw [->] (2, -0.70) -- (2, -0.701) node [right] {$\cev{e}_4$}; \draw [->] (-1.5, 1) -- (-1.51, 1) node [above] {$\vec{e}_2$}; \draw [->] (-2, -0.70) -- (-2, -0.701) node [left] {$\vec{e}_3$}; \draw [->] (1.5, -1) -- (1.51, -1) node [below] {$\vec{e}_4$}; \draw [->] (-0.5, -1) -- (-0.51, -1) node [above] {$\cev{e}_3$}; \draw [red, ->] (1.51, 0) -- (1.5, 0) node [below] {$\vec{f}_1$}; \draw [red, ->] (2.5, 0) -- (2.51, 0);\draw [red] (3,0) node [right] {$\cev{f}_4$};\draw [red, ->] (0, 0.51) -- (0, 0.5) node [left] {$\vec{f}_2$}; \draw [red, ->] (-0.6, -1.6) -- (-0.651, -1.6);  \draw [red] (-1.2, -1.7) node {$\vec{f}_3$}; \draw [red, ->] (-2.5, 0) -- (-2.51, 0) node [above] {$\vec{f}_4$}; \draw [->] (0.5, 1) -- (0.51, 1) node [below] {$\cev{e}_1$}; \draw [red, ->] (0, 1.3) -- (0, 1.31) node [left] {$\cev{f}_1$}; \draw [red, ->] (-1.41,0) -- (-1.4, 0) node [below] {$\cev{f}_3$}; \draw [->] (-2, 0.6) -- (-2, 0.61) node [right] {$\cev{e}_2$}; \draw [red, ->] (0, -0.51) -- (0, -0.5) node [right] {$\cev{f}_2$};
\end{tikzpicture}
\end{center}
\caption{The graph $G$.}\label{fig:8}
\end{figure}
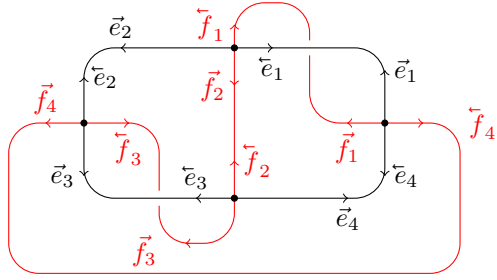
\noindent We have $\sigma_\infty = (\vec{e}_1, \vec{f}_2, \vec{e}_4, \vec{f}_1, \cev{e}_1, \cev{f}_4, \cev{e}_2, \cev{f}_1)(\vec{e}_2, \cev{f}_3, \cev{e}_3, \vec{f}_4, \cev{e}_4, \vec{f}_3, \vec{e}_3, \cev{f}_2)$. Therefore, the number of boundary components in the fat graph $G$ is two (see Lemma~\ref{lem1}).
\end{proof}
%%%%%%%%%%%%%%% SECTION 5 %%%%%%%%%%%%%%%%%%%%%%%%%%%%

\section{Filling pairs of $F_g$}
Before going to the proof of Theorem~\ref{thm:4}, we prove two lemmas which are essential for the proof of the theorem. The lemmas are the particular cases of Theorem~\ref{thm:4} when $g=2$ and $g=3$. 
\begin{lemma}\label{lem:5.3}
For every $k\ (\geq 2) \in \mathbb{Z}$, there exists a filling pair $(\alpha_k^2, \beta_k^2)$ of $F_2$ such that the complement of $(\alpha_k^2 \cup \beta_k^2)$ in $F_2$ is a union of $k$ topological discs.
\end{lemma}

\begin{proof} Suppose that there exists a filling pair $(\alpha_k^2, \beta_k^2)$ of $F_2$ which satisfies this lemma. Now, consider the graph $G_k^2:=\alpha_k^2 \cup \beta_k^2$ on $F_2$. It follows from the Euler characteristic argument that $G_k^2$ has $k+2$ vertices and $k$ boundary components. The number of standard cycles in $G_k^2$ is two which correspond to $\alpha_k^2$ and $\beta_k^2.$ Conversely, if we have a fat graph as above, then by attaching topological discs along the boundary components, we obtain a closed surface $F_2$ of genus $2$. The standard cycles provide us with a filling pair for the lemma. Therefore, to prove the lemma, it suffices to construct a graph $G_k^2=(E, \sim, \sigma_1, \sigma_0)$ as above for each $k\geq 2$. We consider the two cases below.

\begin{large}
\noindent Case 1.
\end{large}
Let $k$ be an even integer and $k=2n$ for some $n\in \mathbb{N}$. Then the number of vertices is $2m$, where $m=n+1$. Now, the graph $G^2_k$ is described below (see Figure~\ref{fig:11}).
\begin{enumerate}
\item $E=\{\vec{e}_i, \cev{e}_i, \vec{f}_i, \cev{f}_i| i=1, \dots, 2m\}$.

\item $\sim$ is determined by the partition $V=\{v_i|1\leq i \leq 2m\}$, where 
\begin{align*}
v_1 &= \{\vec{e}_1, \vec{f}_1, \cev{e}_{2m}, \cev{f}_{2m}\},\\ v_i&=\{\vec{e}_i, \vec{f}_i, \cev{e}_{i-1}, \cev{f}_{i-1}\}; \ 2\leq i\leq m\; \text{and} \\v_i&= \{\vec{e}_i, \cev{f}_{3m-i},  \cev{e}_{i-1}, \vec{f}_{3m-i+1}\};\  m+1\leq i \leq 2m.
\end{align*}
\item $\sigma_1(\vec{e}_i)=\cev{e}_i$ and $\sigma_1(\vec{f}_i)=\cev{f}_i,\; i=1, 2, \dots, 2m.$

\item $\sigma_0=C_1\cdots C_{2m}$, where $C_1=(\vec{e}_1, \vec{f}_1, \cev{e}_{2m}, \cev{f}_{2m}), C_i=(\vec{e}_i, \vec{f}_i, \cev{e}_{i-1}, \cev{f}_{i-1})$ for $i=2,\dots, m$ and $C_i= (\vec{e}_i, \cev{f}_{3m-i}, \cev{e}_{i-1}, \vec{f}_{3m-i+1})$ for $i=m+1, \dots, 2m$.
\end{enumerate}

\begin{figure}[htbp]
\begin{center}
\begin{tikzpicture}
\draw [rounded corners=4mm] (0.5, 2) -- (5.5, 2) -- (5.5, -2) -- (0.5, -2); \draw [dotted] (0.5, -2) -- (-0.5, -2); \draw [rounded corners=4mm] (-0.5, -2) -- (-3.5, -2) -- (-3.5, 2) -- (-0.5, 2); \draw [dotted] (-0.5, 2) -- (0.5, 2); 
\draw [red,rounded corners=4mm] (4.5, 1.9) -- (4.5, 0.7) -- (6.5, 0.7) -- (6.5, -4) -- (-4.5, -4) -- (-4.5, -0.7) -- (-2.5, -0.7) -- (-2.5, -1.9); \draw [red,rounded corners=4mm] (4.5, 2.1) -- (4.5, 3) -- (3.5, 3) -- (3.5, 1) -- (2.5, 1) -- (2.5, 1.9); \draw [red,rounded corners=4mm] (2.5, 2.1) -- (2.5, 3) -- (1.5, 3) -- (1.5, 1) -- (0.5, 1) -- (0.5, 1.9); \draw [red,rounded corners=4mm] (-0.5, 2.1) -- (-0.5, 3) -- (-2, 3) -- (-2, 0) -- (4,0) -- (4,-3) -- (2.8, -3) -- (2.8, -2.1); \draw [red, rounded corners=4mm] (2.8, -1.9) -- (2.8, -1) -- (1.6, -1) -- (1.6, -3) -- (0.5, -3) -- (0.5, -2.1); \draw [red, rounded corners=4mm] (-0.5, -1.9) -- (-0.5, -1) -- (-1.5, -1) -- (-1.5, -3) -- (-2.5, -3) -- (-2.5, -2.1);

\draw [green, fill] (5.5,0.7) circle [radius=0.04]; \draw [green, fill] (-3.5,-0.7) circle [radius=0.04]; \draw [green, fill] (3.5,2) circle [radius=0.04]; \draw [green, fill] (1.5,2) circle [radius=0.04]; \draw [green, fill] (-2,2) circle [radius=0.04]; \draw [green, fill] (1.6,-2) circle [radius=0.04]; \draw [green, fill] (4,-2) circle [radius=0.04]; \draw [green, fill] (-1.5,-2) circle [radius=0.04];
\draw [->] (5.5, 1.5) -- (5.5, 1.51) node [left] {$\vec{e}_1$}; \draw [->] (4.2, 2) -- (4.21, 2) node [below] {$\cev{e}_1$}; \draw [->] (2.91, 2) -- (2.9, 2) node [above] {$\vec{e}_2$}; \draw [->] (2.2, 2) -- (2.21, 2) node [below] {$\cev{e}_2$}; \draw [->] (0.91, 2) -- (0.9, 2) node [above] {$\vec{e}_3$}; \draw [->] (-1.01, 2) -- (-1, 2) node [below] {$\cev{e}_{m-1}$}; \draw [->] (-2.9, 2) -- (-2.91, 2) node [above] {$\vec{e}_m$}; \draw [->] (-3.5, 0.2) -- (-3.5, 0.21) node [left] {$\cev{e}_m$};  \draw [->] (-3.5, -1.5) -- (-3.5, -1.51) node [right] {$\vec{e}_{m+1}$};  \draw [->] (-2, -2) -- (-2.01, -2) node [below] {$\cev{e}_{m+1}$}; \draw [->] (-1.01, -2) -- (-1, -2) node [above] {$\vec{e}_{m+2}$}; \draw [->] (2.2, -2) -- (2.21, -2) node [below] {$\vec{e}_{2m-1}$}; \draw [->] (5, -2) -- (5.01, -2) node [below] {$\vec{e}_{2m}$}; \draw [->] (0.81, -2) -- (0.8, -2) node [above] {$\vec{e}_{2m-2}$}; \draw [->] (5.5, -0.2) -- (5.5, -0.21) node [right] {$\cev{e}_{2m}$}; \draw [->] (3.41, -2) -- (3.4, -2) node [above] {$\vec{e}_{2m-1}$};

\draw [red, ->] (4.81, 0.7) -- (4.8, 0.7) node [below] {$\vec{f}_1$}; \draw [red, ->] (6.2, 0.7) -- (6.21, 0.7) node [above] {$\cev{f}_{2m}$}; \draw [red, ->] (4, 3) -- (4.01, 3) node [above] {$\cev{f}_1$}; \draw [red, ->] (3.01, 1) -- (3, 1) node [below] {$\vec{f}_2$}; \draw [red, ->] (2, 3) -- (2.01, 3) node [above] {$\cev{f}_2$}; \draw [red, ->] (1.01, 1) -- (1, 1) node [below] {$\vec{f}_3$};  \draw [red, ->] (-1.21, 3) -- (-1.2, 3) node [above] {$\cev{f}_{m-1}$}; \draw [red, ->] (-2, 1.01) -- (-2, 1) node [left] {$\vec{f}_{m}$}; \draw [red, ->] (4, -1.01) -- (4,-1) node [right] {$\cev{f}_{m}$}; \draw [red, ->] (3.41, -3) -- (3.4,-3) node [below] {$\vec{f}_{m+1}$}; \draw [red, ->] (1.11, -3) -- (1.1,-3) node [below] {$\vec{f}_{m+2}$}; \draw [red, ->] (-2, -3) -- (-2.01,-3) node [below] {$\vec{f}_{2m-1}$}; \draw [red, ->] (-4.1, -0.7) -- (-4.11,-0.7) node [below] {$\vec{f}_{2m}$}; \draw [red, ->] (-2.86, -0.7) -- (-2.85,-0.7) node [above] {$\cev{f}_{2m-1}$}; \draw [red, ->] (-1.01, -1) -- (-1,-1) node [above] {$\cev{f}_{2m-2}$}; \draw [red, ->] (2.2, -1) -- (2.21,-1) node [above] {$\cev{f}_{m+1}$};
\end{tikzpicture}
\end{center}
\caption{The graph $G_k^2$\;($k$ is even).}\label{fig:11}
\end{figure}
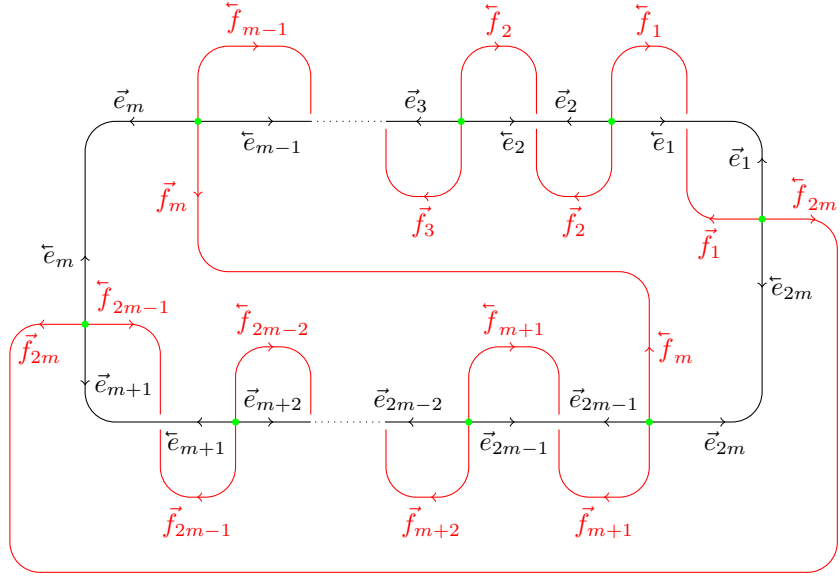

Now, we count the boundary components of $G_k^2$, equivalently orbits of $\sigma_{\infty}.$ The orbits of $\sigma_{\infty}$ are given by 
\begin{eqnarray*}
\partial_i &=& (\vec{e}_i, \vec{f}_{i+1}, \cev{e}_{i+1},  \cev{f}_i)\; \textit{for}\; i=1,\dots, m-2, \\ \partial_{m-1} &=& (\vec{e}_{m-1},\vec{f}_m, \vec{e}_{2m}, \vec{f}_1, \cev{e}_1, \cev{f}_{2m}, \cev{e}_{m}, \cev{f}_{m-1}),\\ \partial_{m}&=&(\vec{e}_m, \cev{f}_{2m-1}, \cev{e}_{m+1}, \vec{f}_{2m}, \cev{e}_{2m}, \vec{f}_{m+1}, \vec{e}_{2m-1}, \cev{f}_m)\; \text{and}\\ \partial_j&=&(\vec{e}_j, \cev{f}_{3m-j-1}, \cev{e}_{j+1}, \vec{f}_{3m-j});\ j=m+1,\dots, 2m-2.
\end{eqnarray*}
So, there are $k$ orbits of $\sigma_{\infty}$.

\begin{large}
\noindent Case 2.
\end{large} In this case, we consider $k=2n+1, n\in \mathbb{N}$. The graph $G_k^2=(E, \sim, \sigma_1, \sigma_0)$ is described below (see Figure~\ref{fig:12}).
\begin{enumerate}
\item $E=\{\vec{e}_i, \cev{e}_i, \vec{f}_i, \cev{f}_i| i=1, \dots, 2n+3\}.$

\item $\sim$ is uniquely determined by the partition $V=\{v_i|i=1,\dots, 2n+3\}$ of $E$, where 
\begin{eqnarray*}
v_1 &=& \{\vec{e}_1, \vec{f}_1, \cev{e}_{2n+3}, \cev{f}_{2n+3} \},\\ v_i &=& \{ \vec{e}_i, \cev{f}_{2n+3-i}, \cev{e}_{i-1}, \vec{f}_{2n+4-i}\};\; i=2,\dots, 2n+2\; \text{and}\\ v_{2n+3} &=& \{\vec{e}_{2n+3}, \vec{f}_{2n+3}, \cev{e}_{2n+2}, \cev{f}_{2n+2}\}.
\end{eqnarray*}

\item $\sigma_1 (\vec{e}_i) = \cev{e}_i$, and $\sigma_1(\vec{f}_i) = \cev{f}_i$ for $ i=1, \dots, 2n+3.$

\item $\sigma_0=C_1C_2\cdots C_{2n+3}$, where $C_1=(\vec{e}_1, \vec{f}_1, \cev{e}_{2n+3}, \cev{f}_{2n+3})$, $C_i =  (\vec{e}_i, \cev{f}_{2n+3-i},\\ \cev{e}_{i-1}, \vec{f}_{2n+4-i})$ for $i=2,\dots, 2n+2$ and $ C_{2n+3}=(\vec{e}_{2n+3}, \vec{f}_{2n+3}, \cev{e}_{2n+2}, \cev{f}_{2n+2})$.

\end{enumerate}
\begin{figure}[htbp]
\begin{center}
\begin{tikzpicture}
\draw [rounded corners=10mm] (-1.8, 2) -- (-4.5,2) -- (-4.5, -2) -- (4.5, -2) -- (4.5, 2) -- (0.2, 2); \draw [dotted] (0.2, 2) -- (-1.8, 2);

\draw [red, rounded corners=2mm] (4.6, -1) -- (5, -1)-- (5,0) -- (-5, 0) -- (-5, 1) -- (-4.6, 1);

\draw [red, rounded corners=2mm] (-4.4, 1) -- (-3.9, 1) -- (-3.9, 3) -- (-3.2,3) -- (-3.2, 2.1);

\draw [red, rounded corners=2mm] (-3.2, 1.9) -- (-3.2, 1) -- (-2.4,1) -- (-2.4, 3) -- (-1.7,3) -- (-1.7, 2.1);
\draw [red, ->] (-2.4, 1.31)--(-2.4, 1.3)node [right] {{\tiny $\cev{f}_3$}};
\draw [red, ->] (-2.4, 2.7)--(-2.4, 2.71)node [left] {{\tiny $\vec{f}_4$}}; \draw [->] (-2.8,2) -- (-2.81,2) node [below] {{\tiny $\vec{e}_{2n}$}}; \draw [->] (-2.1,2) -- (-2,2); \draw (-1.8, 1.8)node {{\tiny $\cev{e}_{2n-1}$}} ;

\draw [red, rounded corners=2mm] (0.1, 1.9) -- (0.1, 1) -- (0.8, 1) -- (0.8, 3) -- (1.5, 3) -- (1.5, 2.1); 
\draw [red, ->] (0.8, 1.15)--(0.8, 1.14); \draw [red] (0.8, 0.7) node {{\tiny $\cev{f}_{2n-1}$}};
\draw [red, ->] (0.8, 2.7)--(0.8, 2.71)node [left] {{\tiny $\vec{f}_{2n}$}}; \draw [->] (0.41,2) -- (0.4,2) node [below] {{\tiny $\vec{e}_4$}}; \draw [->] (1.2,2) -- (1.21,2)node [below] {{\tiny $\cev{e}_3$}} ;

\draw [red, rounded corners=2mm] (1.5, 1.9) -- (1.5, 1) -- (2.3,1) -- (2.3, 3) -- (3.1,3) -- (3.1, 2.1);

\draw [red, ->] (2.3, 1.25)--(2.3, 1.24)node [right] {{\tiny $\cev{f}_{2n}$}};
\draw [red, ->] (2.3, 2.7)--(2.3, 2.71); \draw [red] (2.3, 2.9)node [above] {{\tiny $\vec{f}_{2n+1}$}}; 
\draw [->] (1.91,2) -- (1.9,2) node [below] {{\tiny $\vec{e}_3$}}; 
\draw [->] (2.7,2) -- (2.71,2)node [below] {{\tiny $\cev{e}_2$}} ;

\draw [red, rounded corners=2mm] (3.1, 1.9) -- (3.1, 1) -- (3.8,1) -- (3.8, 3) -- (6,3) -- (6, -3) -- (0, -3) -- (0, -1) -- (4.4, -1);
\draw [red, ->] (3.51, 1)--(3.5, 1) node [below] {{\tiny $\cev{f}_{2n+1}$}}; \draw [red, ->] (3.8, 2.7)--(3.8, 2.71)node [right] {{\tiny $\vec{f}_{2n+2}$}}; \draw [->] (3.41,2) -- (3.4,2) node [below] {{\tiny $\vec{e}_2$}}; \draw [->] (4,1.85) -- (4.1,1.82); \draw (4.4, 1.8)node {{\tiny $\cev{e}_1$}} ; 

\draw [->] (0.5, -2) -- (0.51, -2); \draw (0.9, -1.5)node [below] {{\tiny $\vec{e}_{2n+3}$}};
\draw [->] (-0.5, -2) -- (-0.51, -2); \draw (-0.7, -2.2)node {{\tiny $\cev{e}_{2n+2}$}};
\draw [red, ->] (0, -1.51) -- (0, -1.5)node [left] {{\tiny $\vec{f}_{2n+3}$}};
\draw [red, ->] (0, -2.5) -- (0, -2.51)node [right] {{\tiny $\cev{f}_{2n+2}$}};

\draw [->] (4.5, 0.5) -- (4.5, 0.51)node [right] {{\tiny $\vec{e}_1$}}; \draw [->] (4.5, -0.5) -- (4.5,-0.51)node [left] {{\tiny $\cev{e}_{2n+3}$}}; 

\draw [red, ->] (3.21, 0) -- (3.2,0)node [above] {{\tiny $\vec{f}_1$}}; \draw [red, ->] (5, -0.5) -- (5, -0.51); \draw [red] (5, -0.5) node [right] {{\tiny $\cev{f}_{2n+3}$}};

\draw [->] (-4.5, 0.5) -- (-4.5, 0.51)node [right] {{\tiny $\cev{e}_{2n+1}$}}; \draw [->] (-4.5, -0.5) -- (-4.5,-0.51)node [right] {{\tiny $\vec{e}_{2n+2}$}}; 

\draw [red, ->] (-3.01, 0) -- (-3,0)node [above] {{\tiny $\cev{f}_1$}}; \draw [red, ->] (-5, 0.5) -- (-5, 0.51) node [left] {{\tiny $\cev{f}_2$}};

\draw [->] (-4.2,1.75) -- (-4.21,1.744); \draw (-4.7, 1.8) node {{\tiny $\vec{e}_{2n+1}$}}; \draw [red, ->] (-3.9,1.31)-- (-3.9, 1.3) node [right] {{\tiny $\cev{f}_2$}};  \draw [red, ->] (-3.9, 2.7) -- (-3.9, 2.71)node [left] {{\tiny $\vec{f}_3$}}; \draw [->] (-3.51, 2) -- (-3.5, 2)node [above] {{\tiny $\cev{e}_{2n}$}};
\draw [green, fill] (4.5, 0) circle [radius=0.04]; \draw [green, fill] (-4.5, 0) circle [radius=0.04]; \draw [green, fill] (0, -2) circle [radius=0.04]; \draw [green, fill] (3.81,1.95) circle [radius=0.04]; \draw [green, fill] (2.3,2) circle [radius=0.04]; \draw [green, fill] (0.8,2) circle [radius=0.04]; \draw [green, fill] (-3.89,1.905) circle [radius=0.04]; \draw [green, fill] (-2.4,2) circle [radius=0.04];

\end{tikzpicture}
\end{center}
\caption{The graph $G^2_k$\;($k$ is odd).}\label{fig:12}
\end{figure}
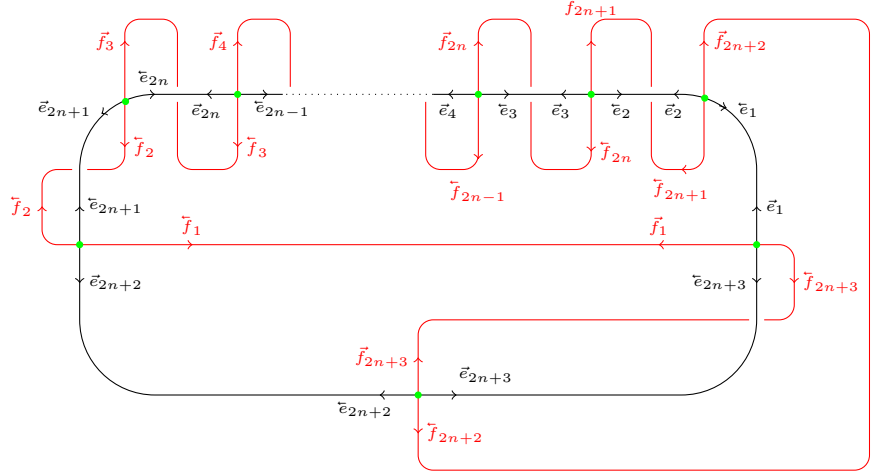
The orbits of $\sigma_{\infty}$ are given by 
\begin{eqnarray*}
D_0&=&(\vec{e}_1, \cev{f}_{2n+1}, \cev{e}_2, \vec{f}_{2n+2}, \cev{e}_{2n+2}, \vec{f}_2, \vec{e}_{2n-1}, \cev{f}_1), \\
D_1&=&(\vec{f}_1, \vec{e}_{2n+2}, \vec{f}_{2n+3}, \cev{e}_{2n+3}, \cev{f}_{2n+2}, \cev{e}_1, \cev{f}_{2n+3}, \vec{e}_{2n+3})
,\\ D_i&=&(\vec{e}_i, \cev{f}_{2k+2-i}, \cev{e}_{i+1}, \vec{f}_{2k+3-i})\;\text{for}\; i=2,\dots, 2n.
\end{eqnarray*}
Hence, the number of orbits of $\sigma_\infty$ is $k$ which is equal to the number of boundary components of the fat graph $G^2_k.$
\end{proof}

\begin{lemma}\label{lem:5.4}
For every $k\ (\geq 1)\in \mathbb{Z}$, there exists a filling pair $(\alpha_k^3, \beta_k^3)$ of $F_3$ such that the complement $F_3\setminus \left(\alpha_k^3\cup \beta_k^3\right)$ is a disjoint union of $k$ topological discs.
\end{lemma}
\begin{proof}
The proof of this lemma is similar to the proof of Lemma~\ref{lem:5.3}. We construct $4$-regular fat graph $G^3_k=(E, \sim, \sigma_1, \sigma_0)$ with $k+4$ vertices, two standard cycles and $k$ boundary components. As before, we consider two cases.

\begin{large}
\noindent Case 1.
\end{large}
In this case, we consider $k$ as an odd integer. Let $k=2m-1$ for some $m\in \mathbb{N}$. The graph is described below.  
\begin{enumerate}
\item $E=\{\vec{e}_i, \vec{f}_i, \cev{e}_i, \cev{f}_i| i=1,\dots, 2m+3\}.$

\item $\sim$ is uniquely determined by the partition $V=\{v_i|i= 1,\dots, 2m+3\}$ of $E$, where 
\begin{eqnarray*}
v_1 &=& \{\vec{e}_1, \vec{f}_1, \cev{e}_{2m+3}, \cev{f}_{2m+3}\}, \\  v_i &=& \{\vec{e}_i, \cev{f}_i, \cev{e}_{i-1}, \vec{f}_{i+1}\}; \; i=2,\dots, 2m+1,\\ v_{2m+2} &=& \{\vec{e}_{2m+2}, \cev{f}_1, \cev{e}_{2m+1}, \vec{f}_2\}\; \text{and}\\ v_{2m+3} &=& \{\vec{e}_{2m+3}, \vec{f}_{2m+3}, \cev{e}_{2m+2}, \cev{f}_{2m+2}\}.
\end{eqnarray*}
 
\item $\sigma_1(\vec{e}_i)=\cev{e}_i$ and $\sigma_1(\vec{f}_i)=\cev{f}_i$ for $i=1,\dots, 2m+3$.

\item $\sigma_0=C_1C_2\cdots C_{2m+3}$, where $C_i$'s are the pairwise disjoint cycles given by $C_1=(\vec{e}_1, \vec{f}_1, \cev{e}_{2m+3}, \cev{f}_{2m+3})$, $C_i=(\vec{e}_i, \cev{f}_i, \cev{e}_{i-1}, \vec{f}_{i+1})$ for $i=2,\dots,2m+1$, $C_{2m+2}=(\vec{e}_{2m+2}, \cev{f}_1, \cev{e}_{2m+1}, \vec{f}_2)$ and $C_{2m+3}=(\vec{e}_{2m+3}, \vec{f}_{2m+3}, \cev{e}_{2m+2}, \cev{f}_{2m+2})$. 
\end{enumerate}
Now, we count the number of orbits of $\sigma_\infty$. The orbit of $\sigma_\infty$ containing $\vec{e}_1$ is $(\vec{e}_1, \cev{f}_2, \cev{e}_{2m+1}, \vec{f}_{2m+2}, \cev{e}_{2m+2}, \vec{f}_2, \vec{e}_2, \cev{f}_3, \cev{e}_1, \cev{f}_{2m+3}, \cev{e}_{2m+3}, \vec{f}_1, \vec{e}_{2m+2}, \vec{f}_{2m+3}, \vec{e}_{2m+3}, \\ \cev{f}_{2m+2}, \cev{e}_{2m}, \vec{f}_{2m+1}, \vec{e}_{2m+1}, \cev{f}_1 )$. Now, for each $i=3, \dots, 2m$, the orbit containing $\vec{e}_i$ is $(\vec{e}_i, \cev{f}_{i+1}, \cev{e}_{i-1}, \vec{f}_i)$. These are the all orbits of $\sigma_{\infty}$. Hence, there are $2m-1=k$ distinct orbits of $\sigma_\infty$.

\begin{large}
\noindent Case 2.
\end{large}
Let $k$ be an even integer and $k=2m$ for some $m\in \mathbb{N}$.  The graph is described below.
\begin{enumerate}
\item $E=\{\vec{e}_i, \vec{f}_i, \cev{e}_i, \cev{f}_i|i=1,\dots,2m+4\}$.
\item $\sim $ is determined by the partition $V=\{v_i|i=1,\dots,2m+4\}$, where 
\begin{eqnarray*}
v_1 &=& \{\vec{e}_1, \vec{f}_1, \cev{e}_{2m+4}, \cev{f}_{2m+4}\},\\  v_i &=& \{\vec{e}_i, \cev{f}_i, \cev{e}_{i-1}, \vec{f}_{i+1}\}; i=2,\dots, m+2,\\ v_{m+3} &=& \{\vec{e}_{m+3}, \cev{f}_1, \cev{e}_{m+2}, \vec{f}_2\}\; \text{and}\\ v_i &=& \{\vec{e}_i, \vec{f}_i, \cev{e}_{i-1}, \cev{f}_{i-1}\}; i=m+4, \dots, 2m+4. 
\end{eqnarray*}

\item $\sigma_1(\vec{e}_i)=\cev{e}_i$ and $\sigma_1(\vec{f}_i)=\cev{f}_i$ for $i=1,\dots,2m+4.$ 

\item $\sigma_0=C_1C_2\cdots C_{2m+4}$, where $C_1=(\vec{e}_1, \vec{f}_1, \cev{e}_{2m+4}, \cev{f}_{2m+4})$, $C_i=(\vec{e}_i, \cev{f}_i, \cev{e}_{i-1},\\ \vec{f}_{i+1})$ for $i=2,\dots, m+2$,    $C_{m+3}=(\vec{e}_{m+3}, \cev{f}_1, \cev{e}_{m+2}, \vec{f}_2)$  and $C_i=(\vec{e}_i, \vec{f}_i, \cev{e}_{i-1},\\ \cev{f}_{i-1})$ for $i=m+4, \dots, 2m+4$. 
\end{enumerate}
The orbit containing $\vec{e}_1$ is given by $(\vec{e}_1, \cev{f}_2, \cev{e}_{m+2}, \vec{f}_{m+3}, \cev{e}_{m+3}, \vec{f}_2, \vec{e}_2,\cev{f}_3, \cev{e}_1, \cev{f}_{2m+4},\\ \vec{e}_{2m+4}, \vec{f}_1, \vec{e}_{m+3}, \vec{f}_{m+4}, \cev{e}_{m+4}, \cev{f}_{m+3}, \cev{e}_{m+1}, \vec{f}_{m+2}, \vec{e}_{m+2}, \cev{f}_1)$. For each $i=3, \dots, m+1,$ the orbit containing $\vec{e}_i$ is $(\vec{e}_i, \cev{f}_{i+1}, \cev{e}_{i-1}, \vec{f}_i)$ and for each $i=m+4, \dots, 2m+3,$ the orbit containing $\vec{e}_i$ is given by $(\vec{e}_i, \vec{f}_{i+1}, \cev{e}_{i+1}, \cev{f}_i)$.  Hence, the number of orbits of $\sigma_\infty$ is $2m=k$. 
\end{proof}

\begin{example}\label{eg:(3,1)}
Consider the $4$-regular fat graph $G$ as given in Figure~\ref{fig:g=3}. It has five vertices, one boundary component and two standard cycles. Therefore, the pair of simple closed curves $(\alpha, \beta)$, associated with the standard cycles of $G$, is a minimal filling pair of the closed surface $\Hat{\Sigma}(G)$ which is obtained by capping $\Sigma(G)$. It follows from the Euler characteristic argument that $\Hat{\Sigma}(G)$ is homeomorphic to $F_3$.
\end{example} 

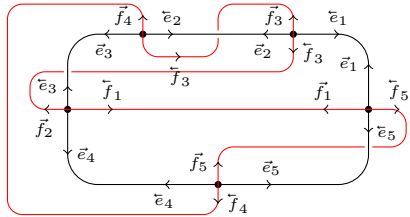
\begin{figure}[htbp]
\begin{tikzpicture}
\draw [rounded corners= 4mm] (2, 1) -- (-2, 1) -- (-2, -1) -- (2, -1) -- cycle; \draw [fill] (0,-1) circle [radius=0.04]; \draw [fill] (1,1) circle [radius=0.04]; \draw [fill] (-1,1) circle [radius=0.04]; \draw [fill] (2, 0) circle [radius=0.04]; \draw [fill] (-2, 0) circle [radius=0.04]; \draw [red, rounded corners = 2mm] (2.05, -0.5) -- (2.5, -0.5) -- (2.5, 0) -- (-2.5, 0) -- (-2.5, 0.5) -- (-2.05, 0.5); \draw [red, rounded corners = 2mm] (-1.95, 0.5) -- (1, 0.5) -- (1, 1.4) -- (0, 1.4) -- (0, 1.05); \draw [red, rounded corners = 2mm] (0, 0.95) -- (0, 0.7) -- (-1, 0.7) -- (-1, 1.4) -- (-2.8, 1.4) -- (-2.8, -1.4) -- (0, -1.4) -- (0, -0.5) -- (1.95, -0.5);  \draw [->] (2, 0.6) -- (2, 0.602) node [left] {\tiny $\vec{e}_1$}; \draw [->] (0.61, 1) -- (0.6,1) node [below]{\tiny $\vec{e}_2$}; \draw [->] (-1.5, 1) -- (-1.51, 1) node [below] {\tiny $\vec{e}_3$}; \draw [->] (-2, -0.6) -- (-2, -0.602) node [right] {\tiny $\vec{e}_4$}; \draw [->] (0.7, -1) -- (0.71, -1) node [above] {\tiny $\vec{e}_5$}; \draw [->] (2, -0.3) -- (2, -0.31) node [right] {\tiny $\cev{e}_5$}; \draw [->] (-0.7, -1) -- (-0.71, -1) node [below] {\tiny $\cev{e}_4$}; \draw [->] (-2, 0.3) -- (-2, 0.31) node [left] { \tiny $\cev{e}_3$}; \draw [->] (-0.61, 1) -- (-0.6,1) node [above] {\tiny $\cev{e}_2$}; \draw [->] (1.6, 1) -- (1.61,1) node [above] {\tiny $\cev{e}_1$}; \draw [->] (1.41, 0) -- (1.4, 0) node [above] {\tiny $\vec{f}_1$}; \draw [->] (-2.3, 0) -- (-2.31, 0) node [below] {\tiny $\vec{f}_2$}; \draw [->] (1, 1.25) -- (1, 1.251) node [left] {\tiny $\vec{f}_3$}; \draw [->] (-1, 1.25) -- (-1, 1.251) node [left] {\tiny $\vec{f}_4$}; \draw [->] (0, -0.71) -- (0, -0.7) node [left] {\tiny $\vec{f}_5$}; \draw [->] (0, -1.25) -- (0, -1.251) node [right] {\tiny $\cev{f}_4$}; \draw [->] (-0.51, 0.7) -- (-0.5, 0.7) node [below] {\tiny $\cev{f}_3$}; \draw [->] (1, 0.751) -- (1, 0.75) node [right] {\tiny $\cev{f}_3$}; \draw [->] (-1.41, 0) -- (-1.4, 0) node [above] {\tiny $\cev{f}_1$}; \draw [->] (2.4, 0) -- (2.41, 0) node [above] {\tiny $\cev{f}_5$};
\end{tikzpicture}

\caption{The graph $G$.}\label{fig:g=3}
\end{figure}

\begin{example}\label{eg:5.1}
In this example, we consider a fat graph $\Gamma=(E, \sim, \sigma_1, \sigma_0)$ given below and compute its boundary components. 
\begin{enumerate}
\item $E=\{\vec{y}_i, \vec{x}_i, \cev{y}_i, \cev{x}_i|\; i=1, \dots, 6\}$.
\item $\sim $ is defined by the partition $U=\{u_i|\; i=1, \dots, 6\}$ of $E$, where $u_1=\{\vec{y}_1, \cev{x}_6, \cev{y}_6, \vec{x}_1\}, u_2=\{\vec{y}_2, \vec{x}_3, \cev{y}_1, \cev{x}_2\}, u_3=\{\vec{y}_3, \vec{x}_2, \cev{y}_2, \cev{x}_1\}, u_4=\{\vec{y}_4, \cev{x}_3, \cev{y}_3, \vec{x}_4\},\\ u_5=\{\vec{y}_5, \vec{x}_6, \cev{y}_4, \cev{x}_5\}$ and $u_6=\{\vec{y}_6, \vec{x}_5, \cev{y}_5, \cev{x}_4\}.$ 
\item $\sigma_1(\vec{x}_i)=\cev{x}_i, \sigma_1(\vec{y}_i)=\cev{y}_i, i=1, \dots, 6$.
\item $\sigma_0=(\vec{y}_1, \cev{x}_6, \cev{y}_6, \vec{x}_1)(\vec{y}_2, \vec{x}_3, \cev{y}_1, \cev{x}_2)(\vec{y}_3, \vec{x}_2, \cev{y}_2, \cev{x}_1)(\vec{y}_4, \cev{x}_3, \cev{y}_3, \vec{x}_4)(\vec{y}_5, \vec{x}_6, \cev{y}_4, \cev{x}_5)\\(\vec{y}_6, \vec{x}_5, \cev{y}_5, \cev{x}_4).$ 
\end{enumerate}
The boundary components of $\Gamma$ are $(\vec{y}_1, \vec{x}_3, \vec{y}_4, \vec{x}_6), (\vec{y}_2, \vec{x}_2, \cev{y}_1, \vec{x}_1, \cev{y}_2, \cev{x}_2, \vec{y}_3, \cev{x}_3),\\ (\vec{y}_5, \vec{x}_5, \cev{y}_4, \vec{x}_4, \cev{y}_5, \cev{x}_5, \vec{y}_6, \cev{x}_6)$ and $(\cev{y}_3, \cev{x}_1, \cev{y}_6, \cev{x}_4)$.
\end{example}
Let $G_k^g$ denote a 4-regular fat graph with two standard cycles satisfying the following.
\begin{enumerate}
\item The number of boundary components is $k$.

\item The surface, obtained by gluing discs along the boundary components, is homeomorphic to $F_g$.
\end{enumerate}
Then the standard cycles of the graph provide us with a filling pair of $F_g$ such that the complement is a disjoint union of $k$ topological discs. It follows from the Euler characteristic argument that the number of vertices in $G_k^g$ is $m=2g-2+k.$
\begin{proof}[Proof of Theorem~\ref{thm:4}]
We prove the theorem by mathematical induction. It is already proved in Lemma~\ref{lem:5.3} and Lemma~\ref{lem:5.4} for the cases when $g=2$ and $g=3$ respectively. Further, for $(g, k)=(3,1)$ the graph $G_k^g$ is given in Example~\ref{eg:(3,1)}.

Suppose $G_k^g=(E, \sim, \sigma_1, \sigma_0)$ is given, we attach the graph $\Gamma$ (see Example~\ref{eg:5.1}) at the vertex $u_1$ with $G_k^g$ at the vertex $v_1$ and obtain $G_k^g\#_{(v_1, u_1)}\Gamma$. Let $E=\{\vec{e}_i, \cev{e}_i, \vec{f}_i, \cev{f}_i|i=1,\dots, 2g-2+k\}$. We label the graph $G_k^g$ such that $v_1=\{\vec{e}_1, \vec{f}_1, \cev{e}_m, \cev{f}_m\}$ is a vertex. Now, the graph $G_k^g\#_{(v_1, u_1)}\Gamma=(E', \sim', \sigma_1', \sigma_0')$ is described below.
\begin{enumerate}
\item $E'=\{\vec{e'}_i, \cev{e}_i', \vec{f}_i', \cev{f'}_i| i=1,2, \dots, m+4\}$, where $\vec{e'}_1=\cev{x}_1*\vec{e}_1$, $$\vec{e'}_i=\vec{e}_i\; \text{for}\; i=2, \dots, m-1,$$ $\vec{e'}_m=\vec{e}_m*\cev{x}_6, \vec{e'}_{m+1}=\cev{x}_5, \vec{e'}_{m+2}=\cev{x}_4, \vec{e'}_{m+3}=\cev{x}_3, \vec{e'}_{m+4}=\cev{x}_2$; $\vec{f}_1'=\vec{y}_6*\vec{f}_1$, $$\vec{f}_i'=\vec{f}_i\; \text{for}\; i=2, \dots, m-1,$$ $\vec{f}_m'=\vec{f}_m*\vec{y}_1, \vec{f}_{m+1}'=\vec{y}_2, \vec{f}_{m+2}'=\vec{y}_3, \vec{f}_{m+3}'=\vec{y}_4$ and $\vec{f}_{m+4}'=\vec{y}_5$.

\item Let $v=\{a_1, a_2, a_3, a_4\}$ be a vertex in $V\cup U\setminus \{v_1, u_1\}$. We define $\tilde{v}=\{a'_1, a'_2, a'_3, a'_4\}$, where
\[  a_j'= \left\{
\begin{array}{ll}
      \cev{x}_1*\vec{e}_1 & \text{if}\;\; a_j\in\{\cev{x}_1, \vec{e}_1\}\\
      a_j & \text{otherwise}\\
\end{array} 
\right. \]
for $j=1,\dots,4$. The set of vertices of $G_k^{g+2}$ is given by $$V'=\{\tilde{v}|v\in V\cup U\setminus \{v_1, u_1\}\}.$$ 
\item $\sigma'_1(\vec{e'}_i)=\cev{e'}_i$ and $\sigma'_1(\vec{f'}_i)=\cev{f'}_i$, where $i=1,2,\dots, m+4$.
\item The permutation $\sigma'_0$ is the product of pairwise disjoint cycles as given below $$\sigma'_0=\prod_{\tilde{v}\in V'}\tilde{v}.$$
\end{enumerate}
Let $\partial$ be a boundary component of $G_k^g$ and $\partial=(a_1, a_2,\dots, a_l)$ which is written uniquely up to cyclic order. We define $\tilde{\partial}$ by following. 

\noindent \textbf{Case 1.} If $a_j\in \{\vec{e}_i, \cev{e}_i, \vec{f}_i, \cev{f}_i|\; i=2, \dots, m-1\}$ for all $j=1, \dots, l$, then we define $\tilde{\partial}=\partial$. 

\noindent \textbf{Case 2.} If $\vec{e}_1$ is in $\partial$, then $\cev{f}_1$ must be in $\partial$ which is counted immediate before $\vec{e}_1$. We replace the subsequence $(\cev{f}_1, \vec{e}_1)$ in $\partial$ by the sequence $(\cev{f}_1*\cev{y}_6, \cev{x}_4, \cev{y}_3, \cev{x}_1*\vec{e}_1)$. Similarly, if $\partial$ contains $(\vec{e}_m, \vec{f}_1), (\vec{f}_m, \cev{f}_m)$ and $(\cev{e}_1, \cev{f}_m)$, then we replace them by $(\vec{e}_m*\cev{x}_6, \vec{y}_5, \vec{x}_5, \cev{y}_4, \vec{x}_4, \cev{y}_5, \cev{x}_5, \vec{y}_6*\vec{f}_1), (\vec{f}_m*\vec{y}_1, \vec{x}_3, \vec{y}_4, \vec{x}_6*\cev{e}_m)$ and $(\cev{e}_1*\vec{x}_1, \cev{y}_2, \cev{x}_2, \vec{y}_3,\\ \cev{x}_3, \vec{y}_2, \vec{x}_2, \cev{y}_1*\cev{f}_m)$ respectively and the obtained new finite sequence of directed edges is defined to be $\tilde{\partial}$.

\noindent Suppose $\{\partial_i| i=1,\dots, k\}$ is the set of boundary components of $G_k^g$. Then $\{\tilde{\partial_i}| i=1,\dots, k\}$ is the set of all boundary components of $G_k^g\#_{(v_1, u_1)}\Gamma.$ Therefore, the number of boundary components in $G_k^g\#_{(v_1, u_1)}\Gamma$ is $k.$ The standard cycles of $G_k^g\#_{(v_1, u_1)}\Gamma$ are $e_1'*e_2'*\cdots*e_{m+4}'$ and $f_1'*f_2'*\cdots*f_{m+4}'$. We define $$G_k^{g+2}:=G_k^g\#_{(v_1, u_1)}\Gamma.$$    
\end{proof}

\end{document}